\documentclass{amsart}
\usepackage[utf8]{inputenc}
\usepackage{amssymb,bm,amsfonts, stmaryrd, amsmath, amsthm, enumerate, mathtools,comment}
\usepackage[all]{xy}
\makeatletter
\@namedef{subjclassname@2020}{\textup{2020} Mathematics Subject Classification}
\makeatother
\usepackage[colorlinks]{hyperref}
\hypersetup{
 colorlinks=true,
 linkcolor=blue,
 filecolor=magenta, 
 urlcolor=cyan,
}
\usepackage{wrapfig,tikz, tikz-cd}
\usetikzlibrary{arrows}
\usepackage[capitalise]{cleveref}
\crefformat{equation}{(#2#1#3)}
\crefrangeformat{equation}{(#3#1#4--#5#2#6)}
\crefformat{enumi}{(#2#1#3)}
\crefrangeformat{enumi}{(#3#1#4--#5#2#6)}
\newtheorem{theorem}{Theorem}[section]
\newtheorem{lemma}[theorem]{Lemma}
\newtheorem{proposition}[theorem]{Proposition}
\newtheorem{corollary}[theorem]{Corollary}
\newcounter{intro}

\newtheorem{introthm}[intro]{Theorem}

\theoremstyle{definition}
\newtheorem{definition}[theorem]{Definition}
\newtheorem{example}[theorem]{Example}
\newtheorem{remark}[theorem]{Remark}

\newtheorem{chunk}[theorem]{}

\newtheorem*{ack}{Acknowledgements}

\theoremstyle{definition}

\newtheorem{setup}{Setup}

\newcommand{\Spec}{{\operatorname{Spec}}}

\newcommand{\Hom}{{\operatorname{Hom}}}

\newcommand{\f}{\bm{f}}
\newcommand{\V}{{\rm{V}}}

\DeclareMathOperator{\id}{id}
\DeclareMathOperator{\height}{height}

\DeclareMathOperator{\rank}{rank}

\newcommand{\m}{\mathfrak{m}}
\newcommand{\p}{\mathfrak{p}}

\DeclareMathOperator{\onto}{\twoheadrightarrow}

\DeclareMathOperator{\Der}{Der}
\DeclareMathOperator{\Sing}{Sing}

\newcommand{\F}{\mathbb{F}}

\newcommand{\Q}{\mathbb{Q}}
\newcommand{\C}{\mathbb{C}}
\newcommand{\Z}{\mathbb{Z}}

\newcommand{\q}{\mathfrak q}
\newcommand{\J}{\mathcal{J}}
\newcommand{\ma}{\mathfrak a}

\title[A Jacobian Criterion in Ramified Mixed Characteristic]{Singular loci of algebras over ramified discrete valuation rings}

\author[N.~KC]{Nawaj KC}
\address{Department of Mathematics,
University of Nebraska, Lincoln, NE 68588, U.S.A.}
\email{nkc3@huskers.unl.edu}


\begin{document}

\begin{abstract}
    \noindent When $k$ is a field, the classical Jacobian criterion computes the singular locus of an equidimensional, finitely generated $k$-algebra as the closed subset of an ideal generated by appropriate minors of the so-called Jacobian matrix. Recently, Hochster-Jeffries and Saito have extended this result for algebras over any unramified discrete valuation ring of mixed characteristic via the use of $p$-derivations. Motivated by these results, in this paper, we state and prove an analogous Jacobian criterion for algebras over ramified discrete valuation rings of mixed characteristic.
\end{abstract}

\maketitle

\section*{Introduction}\label{s_intro}
\noindent Suppose $R$ is a finitely generated, equidimensional $\C$-algebra. We consider the set $\Sing R = \{ \p \in \Spec R \mid R_{\p} \text{ is singular} \}$. A local ring is called singular if it is not a regular local ring, i.e. the minimal number of generators of the maximal ideal is strictly greater than the Krull dimension of the ring. Under our hypothesis, $\Sing R$ is a closed subset of $\Spec R$ and, in fact, we can explicitly compute it. We may assume $R = S/I$ where $S = \C[x_1, \ldots, x_n]$ and $I$ is an ideal ${I = (\f)= (f_1, \ldots, f_t) \subseteq S}$ with $h = \height I$. The \textit{Jacobian matrix} on the set $\f$ is defined as follows.  \[ \J(\mathbf{f}) = \begin{pmatrix}
 \frac{ \partial f_1}{\partial x_1} &\ldots & \frac{ \partial f_t}{\partial x_n} \\
\vdots & \ddots &\vdots \\
\frac{ \partial f_1}{\partial x_n}&\ldots &\frac{\partial f_t}{\partial x_n}
\end{pmatrix}. \]
\noindent Let $I_h(\J(\mathbf{f}))$ denote the ideal generated by $h \times h$ minors of $\J(\mathbf{f})$ in $R$. By the classical Jacobian criterion \cite[Ch. 16]{Eisenbud:1995}, we have the equality: \[ \Sing R = \V(I_{h}(\J(\mathbf{f}))). \]This theorem which goes back to Zariski \cite{Zariski:1947} holds with suitable modifications for all finitely generated, equidimensional $k$-algebras for $k$ any field \cite{Samuel:1955, Nagata:1957, Brezuleanu/Radu:1986}. However, this result does not naively extend to algebras of finite type over discrete valuation rings of mixed characteristic $(0, p)$ where $p$ is a prime integer. For example, consider $V = \Z_p[x]/(x^2-p)$ which is isomorphic to a ramified discrete valuation ring $\Z_p[\sqrt{p}]$ with uniformizer $\sqrt{p}$. Thinking of it as a $\Z_p$ algebra, the ideal of one by one minors of the naive Jacobian matrix $[\frac{\partial}{\partial x}(x^2-2)]$ is $(2x)$. So if the Jacobian criterion held, $(x)$ would be a singular prime, which it is not. 

Suppose $R = S/I$ is equidimensional where $S = \Z_p[x_1, \ldots, x_n]$ and $I = (f_1, \ldots, f_t)$ with $p \not\in I$. Let $X = \Spec R$. If $\p \in X$ such that $p \not\in \p$, then $R_{\p}$ has $p$ inverted, whence we reduce to the case of an algebra over the field $\Q_p = \Z_p[p^{-1}]$. So the usual Jacobian criterion suffices. Hence \[\Sing R = (\Sing R \setminus V(p)) \cup (\Sing R \cap V(p)).\]The first component is equal to ${\V(I_h(\J(\f))) \setminus \V(p)}$ so the problem is to compute the component ${\Sing R \cap V(p)}$. For every prime $\p$ in this component, $R_{\p}$ is a ring such that char $R_{\p} = 0$ but char $ R_{\p}/\p R_{\p} = p$, so we are in the mixed characteristic proper. Due to the recent results of Hochster-Jeffries \cite{Hochster/Jeffries:2021} and Saito \cite{Saito:2022} independently, we can now compute $\Sing R \cap V(p)$. The new ingredient is what is called a $p$-derivation. Going mod $p$, $S$ reduces to $\F_p[x_1, \ldots, x_n]$. The Frobenius on $S/pS$ is uniquely determined by the assignment $x_i \mapsto x_i^p$ for all $i$. Therefore, for a polynomial $f(x_1, \ldots, x_n) \in S$, $f(x_1^p, \ldots, x_n^p) - f(x_1, \ldots, x_n)^p$ is divisible by $p$. Since $S$ is also $p$-torsion free, $\delta_p : S \to S$ defined below is a well-defined function: \[ \delta_{p}(f(x_1, \ldots, x_n)) = \frac{f(x_1^p, \ldots, x_n^p) - f(x_1, \ldots, x_n)^p}{p}. \]
The key property of this operator is that it decreases the $p$-adic order of an element. By the formula for $\delta_p$ given above, for an element $s \in (p^n)$, one can check that $\delta_p(s) \in (p^{n-1})\setminus (p^n)$. In this sense, $\delta_p$, called a $p$-derivation, acts like a ``derivative with respect to $p$" (see also \cite{DeStefani/Grifo/Jeffries:2020, DeStefani/Grifo/Jeffries:2022}). Now one defines the following \textit{mixed} Jacobian matrix.  \[ \widetilde \J(\mathbf{f})) = \begin{pmatrix}
\delta_p(f_1) & \ldots & \delta_p(f_t) \\
\displaystyle(\frac{ \partial f_1}{\partial x_1})^p &\ldots & \displaystyle (\frac{ \partial f_t}{\partial x_1})^p \\
\vdots & \ddots &\vdots \\
\displaystyle (\frac{ \partial f_1}{\partial x_n})^p&\ldots & \displaystyle (\frac{\partial f_t}{\partial x_n})^p
\end{pmatrix}. \]
\noindent The recent results of Hochster-Jeffries and Saito states that \[ \Sing R \cap \V(p) = \V(I_{h}(\widetilde \J(\mathbf{f}))) \cap \V(p). \]Armed with this result, let us now go back to the example $V = \Z_p[x]/(x^2-p)$. One can now compute that the ideal of one by one minors of the mixed Jacobian matrix is the unit ideal whence $\Sing V \cap \V(p) = \varnothing$ as it ought to be. Furthermore, this mixed Jacobian criterion of Hochster-Jeffries and Saito generalizes to all finitely generated, equidimensional $V$-algebras where $V$ is an unramified discrete valuation ring of mixed characteristic. This is due to the fact that all unramified discrete valuation rings admit a $p$-derivation modulo $p^2$ (see \cite[Proposition 2.9]{Hochster/Jeffries:2021}) which suffices to state and prove a mixed Jacobian criterion. 

In contrast, if $V$ is a ramified discrete valuation ring, $V$ cannot admit a $p$-derivation modulo $p^2$, so the theorem of Hochster-Jeffries and Saito cannot be directly extended to algebras of finite type over a ramified discrete valuation ring of mixed characteristic; see \cref{HJwork} and \cite[Example 2.5]{DeStefani/Grifo/Jeffries:2020}. In this paper, we state and prove a Jacobian criterion for such rings. Assume  $V$ is a ramified discrete valuation ring of mixed characteristic $(0, p)$ with a perfect residue field $k$ and set $S = V[x_1, \ldots, x_n]$.  Our key observation is that the ramified discrete valuation rings admit an honest $\Z$-derivation that plays an analogous role to $\delta_p(-)$ in the ramified setting. More precisely, we show that there exists an additive derivation $ \frac{\partial}{\partial \pi} : S \to S/\pi S$ such that $\frac{\partial(\pi)}{\partial \pi} = 1$; see Corollary \ref{d/dpi}. Analogously, we then define the following \textit{mixed} Jacobian matrix over $S/\pi S$: \[ \J^{\pi}(\mathbf{f}) = \begin{pmatrix}
\frac{\partial (f_1)}{\partial \pi} & \ldots & \frac{\partial (f_t)}{\partial \pi} \\[0.1cm]
\frac{ \partial f_1}{\partial x_1} &\ldots & \frac{ \partial f_t}{\partial x_1} \\
\vdots & \ddots &\vdots \\
\frac{ \partial f_1}{\partial x_n}&\ldots & \frac{\partial f_t}{\partial x_n}
\end{pmatrix}. \]

\begin{introthm} (\cref{JacobianCriterion})
    Suppose $(V, \pi)$ is a ramified discrete valuation ring of mixed characteristic $(0, p)$ with perfect residue field. Let $S = V[x_1, \ldots, x_n]$ with an ideal $I = (\mathbf{f})= (f_1, \ldots, f_t) \subseteq S$. Assume $R = S/I$ is equidimensional and let $h = \height I$. Then \[ \Sing R \cap \V(p) = \V(I_{h}( \J^{\pi}(\mathbf{f}))) \subseteq \Spec(R/(\pi) R)\] where we identify $V(\pi) = V(p) \subseteq \Spec R$ with $\Spec(R/(\pi) R)$ and $I_h(\J^{\pi}(\mathbf{f})))$ denotes the ideal generated by $h \times h$ minors of $ \J^{\pi}(\mathbf{f})$ in $R/(\pi)R$. 
\end{introthm}

\noindent For $R$ as in the theorem above, set $\overline R = R/(\pi)R$. We will show that this mixed Jacobian matrix above is a presentation for the universal module, denoted $\overline \Omega_{R/\Z}$, that represents the functor $\Der_{\Z}(R, -): \overline R\text{-Mod} \to \overline R\text{-Mod}$. Below, we say a field $k$ of characteristic $p > 0$ is $F$-finite if the Frobenius map on $k$ is a finite map. 

\begin{introthm} (\cref{non-free locus})
    Let $(V, \pi)$ be a ramified discrete valuation ring of mixed characteristic $(0, p)$ with uniformizer $\pi$ such that $k = V/\pi V$ is $F$-finite. Suppose $(R, \m, K)$ is a local algebra of essentially finite type over $V$ such that $p \in \m$. Set $\overline R = R/(\pi)R$ and $\overline \Omega_{R/\Z} = \Omega_{R/\Z}/(\pi)\Omega_{R/\Z}$. Then $R$ is a regular local ring if, and only if, $\overline \Omega_{R/\Z}$ is a free $\overline R$ module of rank $\dim R + \log_p[K : K^p]$. 
\end{introthm}

\begin{ack}
    We thank Melvin Hochster and Jack Jeffries for providing us key insights for this work. We are also grateful to Mark Walker and  Elo\'isa Grifo for fruitful discussions, and the referee for many suggestions. The author was supported through NSF grants DMS-2044833 and DMS-2200732. This material is also based upon work supported by the National Science Foundation under Grant No. DMS-1928930 and by the Alfred P. Sloan Foundation under grant G-2021-16778, while the author was in residence at the Simons Laufer Mathematical Sciences Institute (formerly MSRI) in Berkeley, California, during the Spring 2024 semester.
\end{ack}

\section{Detecting nonsingularity} \label{sec1}

\begin{chunk}
    \textit{Some recollections on derivations:} Fix a pair of commutative rings ${A \subseteq R}$ with $1$ and an $R$-module $M$. An $A$-linear map $d : R \to M$ such that ${d(rs) = rd(s) + sd(r)}$ for all $r, s \in R$ is an $A$-derivation from $R$ to $M$. The set of all such derivations, denoted $\Der_A(R, M)$, is an $R$-module under the rules ${(d+d')(r) = d(r) + d'(r)}$ and ${(r\cdot d)(s) = rd(s)}$ for $d, d'$ derivations and $r, s \in R$. Furthermore, note that given an $R$-linear map $f: M \to N$, we have an $R$-linear map ${\Der_A(R, M) \to \Der_A(R, N)}$ via post-composition, i.e. ${d \mapsto f \circ d}$. This defines an additive functor \[ {\Der_A(R, -) : R\text{-Mod} \to R\text{-Mod}}\] and it is representable. We have the universal $A$-derivation ${d_{R/A}: R \to \Omega_{R/A}}$ from $R$ to the $R$-module of K\"ahler differentials $\Omega_{R/A}$ which induces a natural isomorphism of functors $\Hom_R(\Omega_{R/A}, -)\cong \Der_A(R, -) $  where $(f: \Omega_{R/A} \to M)$ maps to a derivation $ {(f \circ d_{R/A}: R \to M)}$ \cite[Chapter 16]{Eisenbud:1995}. 

\begin{example} \label{polyex}
    Suppose $R = A[x_1, \ldots, x_n]$. Then we have $\Omega_{R/A} = \bigoplus_{i=1}^n R \cdot dx_i$ and $d_{R/A}: R \to \Omega_{R/A}$ is the derivation $f \mapsto \sum_{i=1}^n \frac{\partial(f)}{\partial x_i}dx_i$ \cite[Proposition 16.1]{Eisenbud:1995}. 
\end{example}

\end{chunk}
\begin{chunk}

\noindent \textit{Derivations to modules killed by an ideal: } Suppose $R$ is a commutative ring and we have some ideal ${\mathfrak a} \subseteq R$. We want to consider $A$-linear derivations to $R$-modules that are killed by ${\mathfrak a}$. This turns out to be useful in computing ${\Sing(R) \cap \V({\mathfrak a})}$. Let $ \overline{R} = R/{\mathfrak a}$ and $\overline \Omega_{R/A} = \overline{R} \otimes_R \Omega_{R/A} \cong \Omega_{R/A}/{\mathfrak a}\Omega_{R/A}$. If $M$ is an $R$-module such that $\ma M =0$, via Hom-Tensor adjunction along $R \to \overline R$, we have isomorphisms $\Der_A(R, M) \cong \Hom_R(\Omega_{R/A}, M) \cong \Hom_{\overline R}(\overline{R} \otimes_R \Omega_{R/A}, M)$. In other words, \[ \overline d_{R/A} : R \xrightarrow{d_{R/A}} \Omega_{R/A} \to \overline \Omega_{R/A} \] represents the functor \[ \Der_A(R, -) : \overline R\text{-Mod} \to \overline R\text{-Mod}. \]Note also that for any derivation $d: R \to M$ we have $d({\mathfrak a}^n) \subseteq \mathfrak{a}^{n-1}M$ for all $n \geq 1$. If ${\mathfrak a} M = 0$, $d$ factors through $R/{\mathfrak a}^2$. So $\Der_A(R, M) = \Der_A(R/{\mathfrak a}^2, M)$ for all $R$-modules such that ${\mathfrak a} M = 0$. Therefore, $\overline d_{R/A}$ equivalently represents the functor $\Der_A(R/{\mathfrak a}^2, -): \overline R\text{-Mod} \to \overline R\text{-Mod}$.

\begin{example}
    If $\Z \subseteq R = \Z[x_1, \ldots, x_n]$, then ${\Omega_{R/\Z} \cong \bigoplus_{i=1}^n R\cdot dx_i}$ by Example \ref{polyex}. Suppose $\ma = (p)$ for some prime integer $p \in \Z$. Then ${\overline \Omega_{R/\Z} \cong \bigoplus_{i=1}^n \overline {R}\cdot dx_i}$ and ${\overline d_{R/\Z} (f) = \sum_{i=1}^n \frac{\partial f}{\partial x_i} dx_i}$ modulo $(p)$. Note ${\overline{R} = \F_p[x_1, \ldots, x_n]}$ has the same universal derivation. Therefore, when $\ma = (p)$ derivations from $R$ to $\overline R$-modules is the same as derivations from $\overline R$ itself. The situation is different when ${\mathfrak a} = (x_1)$. We have $\overline \Omega_{R/\Z} \cong \bigoplus_{i=1}^n  \overline R \cdot dx_i$. But $\Omega_{\overline R / \Z} = \bigoplus_{i=2}^n \overline R \cdot dx_i$ as  $\overline R = \Z[x_2, \ldots, x_n]$. So $\overline \Omega_{R/\Z} \not\cong \Omega_{\overline{R}/\Z}$. We have strictly more derivations from $R$ to $\overline R$ modules compared to derivations from $\overline R$ itself. 
\end{example}
\noindent Our proofs crucially use the following basic and fundamental result on derivations. For the convenience of the reader, we include a proof. 
\begin{proposition}\label{thm1}
   Consider rings maps $A \to S \onto S/I$.  For all $S/I$-modules $M$, we have a left exact sequence \[ 0 \to \Der_A(S/I, M) \to \Der_A(S, M) \to \Hom_{S/I}(I/I^2, M) \]This sequence is split exact for all $S/I$-modules $M$ if $S/I^2 \onto S/I$ splits as a map of $A$-algebras.
\end{proposition}
     
\begin{proof}
    Suppose $M$ is an $S/I$-module. An $A$-linear derivation $d: S/I \to M$ maps to an $A$-linear derivation $d \circ p: S \to M$ where $p: S \to S/I$ is the natural surjection. Clearly, if $d \circ p = 0$, then $d = 0$. On the other hand, given a derivation $d: S \to M$, we have $d(I^2) \subseteq IM = 0$, so it factors uniquely through $S/I^2$. Therefore, $\Der_A(S, M) = \Der_A(S/I^2, M)$. Restricting the map $d: S/I^2 \to M$ to the submodule $I/I^2$, it is easy to verify it is an $S/I$-linear map. Hence, we have a left exact sequence as claimed. Now assume $\sigma : S/I^2 \onto S/I$ splits as a map of $A$-algebras and let $\iota: S/I \to S/I^2$ denote the splitting, i.e. $\sigma \circ \iota = \id_{S/I}$. Let $\alpha = \iota \circ \sigma - \id_{S/I^2}$. Note that since $\sigma \circ \alpha = 0$, $\alpha$ is an $A$-linear map to $I/I^2$. On the other hand, as $\alpha + \id_{S/I^2} = \iota \circ \sigma$ is a ring homomorphism, one can verify that $\alpha: S/I^2 \to I/I^2$ must be a derivation. Finally, noting $\Der_A(S, M) = \Der_A(S/I^2, M)$, we may now define a map $\phi: \Hom_{S/I}(I/I^2, M) \to \Der_A(S, M)$ where $f \mapsto f \circ -\alpha$. Finally, $\phi$ splits the map $\Der_A(S, M) \to \Hom_{S/I}(I/I^2, M)$ as $(f \circ -\alpha) (I/I^2) = f$. 
\end{proof}

\begin{proposition} \label{2ndseq}
    Consider ring maps $A \to S \onto S/I = R$. Consider an ideal $\ma_S \subseteq S$ and let $\ma\subseteq R$ denote its corresponding image in $R$. Let $\overline S = S/\ma_S$ and $\overline R = R/\ma$ and consider the induced map $\overline S \to \overline R$. There is an exact sequence of $\overline R$-modules \[ \frac{I}{(I^2+ \ma_S I)} \xrightarrow{\J} \overline R \otimes_{\overline S} \overline\Omega_{ S/A} \to \overline\Omega_{ R/A} \to 0\]where $\J([f]) = 1 \otimes \overline d_{S/A}(f)$ for $f \in I$ and $\overline d_{S/A}: S \to \overline \Omega_{S/A}$. Furthermore, if $S/I^2 \onto S/I$ splits as a map of $A$-algebras, this sequence is split exact. 
\end{proposition}

\begin{proof}
   Given an $\overline R$ module $M$, we have by Proposition \ref{thm1}, an exact sequence $0 \to \Der_A(R, M) \to \Der_A(S, M) \to \Hom_R(I/I^2, M)$. If $M$ is an $\overline R$-module, $\ma M = 0$ and consequently $\ma_S M = 0$. Then we have by the universality of K\"ahler differentials and Hom-Tensor adjunction the following commutative diagram: \[\begin{tikzcd}
	0 & {\Der_A(R, M)} & {\Der_A(S, M)} & {\Hom_R(I/I^2, M)} \\
	0 & {\Hom_{\overline R}(\overline \Omega_{R/A}, M)} & {\Hom_{\overline R}(\overline R \otimes_{\overline S}\overline \Omega_{S/A}, M)} & {\Hom_{\overline R}(I/(I^2+\ma_S I), M)}
	\arrow[from=1-1, to=1-2]
	\arrow[from=1-2, to=1-3]
	\arrow[from=1-3, to=1-4]
	\arrow[from=2-1, to=2-2]
	\arrow[from=2-2, to=2-3]
	\arrow[from=2-3, to=2-4]
	\arrow["\cong", from=1-2, to=2-2]
	\arrow["\cong", from=1-3, to=2-3]
	\arrow["\cong", from=1-4, to=2-4]
\end{tikzcd}\]
Since this holds for all $\overline R$-modules $M$, we have an exact sequence of $\overline R$-modules as claimed. Furthermore, if  $S/I^2 \onto S/I$ splits as a map of $A$-algebras, then by Proposition \ref{thm1} the upper sequence in the diagram is split exact, and hence the exact sequence we get is also split exact. 
\end{proof}

\begin{corollary} \label{cor1}
    Consider an inclusion $A \subseteq S$ of rings and let $\ma_S$ be some ideal of $S$. Set $\overline S = S/ \ma_S$. Then for $\q \in \Spec S$ such that $\q \supseteq \ma_S$, we have a map of $k(\q) = S_{\q}/\q S_{\q}$ modules \[ \frac{\q S_{\q}}{\q^2S_{\q}} \xrightarrow{1 \otimes \overline d_{S/A}} k(\q) \otimes_{\overline S} \overline \Omega_{S/A}\] where $[f] \mapsto 1 \otimes \overline d_{S/A}(f)$ for $f \in \q$. Furthermore, if $S_{\q}/\q^2 S_{\q} \onto S_{\q}/\q S_{\q}$ splits as a map of $A$-algebras, this map is split injective. 
\end{corollary}
\begin{proof}
    Applying Proposition \ref{2ndseq} to maps $A \to S_{\q} \to S_{\q}/\q S_{\q}$, we have a map \[ \q S_{\q}/\q^2 S_{\q} \to k(\q) \otimes_{\overline S} \overline \Omega_{S_{\q}/A} \] where $[f] \mapsto 1 \otimes \overline d_{S_{\q}/A}(f)$ and it is split injective if $S_{\q}/\q^2 S_{\q} \onto S_{\q}/\q S_{\q}$ splits as a map of $A$-algebras. Now note that $k(\q) \otimes_{\overline S} \overline \Omega_{S/A} \cong k(\q) \otimes_{\overline S_{\q}} \overline S_{\q} \otimes_{\overline S} \overline \Omega_{S/A} \cong k(\q) \otimes_{\overline S_{\q}} \overline \Omega_{S_{\q}/A}$ and the universal map $\overline d_{S_{\q}/A}$ is the universal map $\overline d_{S/A}$ localized at $\q$. Therefore the assertion follows. 
\end{proof}
\end{chunk}

\begin{chunk}

\noindent {\em Detecting singularity:} 

\begin{setup} \label{setup}
    Fix a ring $A$ and let $V$ be some $A$-algebra and $S = V[x_1, \ldots, x_n]$. Given an ideal $\ma_S \subseteq S$, we will assume that for all prime ideals $\q $ containing $\ma_S$ the following two conditions are met. 
\begin{enumerate}
    \item $S_{\q}$ is a regular local ring. 
    \item $\frac{S_{\q}}{{\q}^2S_{\q}} \onto \frac{S_{\q}}{{\q}S_{\q}}$ splits as a map of $A$-algebras.
\end{enumerate}
\end{setup}

\begin{example}
    If $A = V = k$ is a perfect field, then $S = k[x_1, \ldots, x_n]$ satisfies the conditions of Setup \ref{setup} for any ideal $\ma_S  \subseteq S$.  The main example in this paper is when $A = \Z$ and $V$ is a ramified discrete valuation ring of mixed characteristic $(0, p)$; see Proposition \ref{mainEx} which we will discuss in Section \ref{ramSec} below. 
\end{example}

\begin{theorem} \label{detectingNonsingularity}
    Suppose $S$ is as in the Setup \ref{setup} above and let $R = S/I$ for some ideal $I \subseteq S$. Suppose $\q \in \Spec S$ such that $\q \supseteq I + \ma_S$ and let $k(\q) := S_{\q}/\q S_{\q}$. Then $R_{\q}$ is a regular local ring if, and only if, $\rank_{k(\q)}(\J \otimes_{\overline R} k(\q)) = \height IS_{\q}$ where $\J$ is the $\overline R$-linear map $\frac{I}{I^2+I\ma} \xhookrightarrow{\J} \overline R \otimes_{\overline S} \overline \Omega_{S/A}$ as in Proposition \ref{2ndseq}. 
\end{theorem}
\begin{proof}
    Note as $\q \supseteq I + \ma_S$, $k(\q)$ is indeed an $\overline R$ module. Also note that \[ \frac{I}{I^2+I\ma} \otimes_{\overline R} k(\q) \cong \frac{I}{I^2+I\ma} \otimes_{\overline S} k(\q) \cong IS_{\q}/I\q S_{\q} \] Hence ${\J \otimes_{\overline R} k(\q) :  IS_{\q}/I\q S_{\q} \to k(\q) \otimes_{\overline S} \overline \Omega_{S/A}}$. We then have a commutative triangle \[\begin{tikzcd}
	{IS_{\mathfrak q}/I{\mathfrak q}S_{{\mathfrak q}}} && {k(\q) \otimes_{\overline S} \overline \Omega_{S/A}} \\
	& {\q S_{\q}/\q^2S_{\q}}
	\arrow["{\J \otimes_{\overline R} k(\q)}", from=1-1, to=1-3]
	\arrow["{\eta_{\q}}"', from=1-1, to=2-2]
	\arrow["{1 \otimes \overline d_{S/A}}"', hook, from=2-2, to=1-3]
\end{tikzcd}\]where $\eta_{\q}$ is the natural map induced by the inclusion $I \subseteq \q$. That this triangle commutes and that $1 \otimes \overline d_{S/A}$ is an injection follows from Corollary \ref{cor1} since  $\frac{S_{\q}}{{\q}^2S_{\q}} \onto \frac{S_{\q}}{{\q}S_{\q}}$ splits as a map of $A$-algebras. So $\rank_{k(\q)}(\J \otimes_{\overline R} k(\q)) = \rank \eta_{\q}$. Since we also assume that $S_{\q}$ is a regular local ring, the proof is complete by the following well known lemma.  
\end{proof}
\begin{lemma}
    Suppose $(A, \m, k)$ is a regular local ring and $J \subseteq A$ is an ideal of height $h$. We have $A/J$ is regular if, and only if, $\rank(J/\m J \to \m/\m^2) = h$.  
\end{lemma} 
\begin{proof}
    Suppose $J$ is minimally generated by $f_1, \ldots, f_t$ where $h \leq t$. Observe that $\rank(J/\m J \to \m/\m^2) = h$ if, and only if, $J$ is minimally generated by ${\f = f_1, \ldots, f_h}$, where $\f$ is a regular $A$-sequence with $f_i \in \m \setminus \m^2$ for all $i$. The latter holds if, and only if, $A/J$ is regular. The only if part is clear. Conversely, if $A/J$ is regular, $\dim_k(J/\m J) = \dim_k(\m/\m^2) - \dim_k(\m/(J+\m^2)) = \height J$, whence $t = h$ and $\f$ is as claimed.

\end{proof}
\end{chunk}

\section{Derivations over a ramified discrete valuation ring} \label{ramSec}

\noindent In this section, we fix $(V, \pi)$ to be a discrete valuation ring of mixed characteristic $(0, p)$ with a uniformizer $\pi$ and $S = V[x_1, \ldots, x_n]$. Let $\Lambda = \{ \lambda : \lambda \in V \}$ be a subset of $V$ that maps bijectively to a $p$-base of $k : = V/(\pi)$. We have that $\Omega_{k/k^p}$ is a free $k$-module with basis $\{ d \lambda : \lambda \text{ $p$-basis element of $k$ over $k^p$}\}$ \cite[Theorem 16.14]{Eisenbud:1995}. We denote by $\frac{\partial}{\partial \lambda}$ the corresponding $k^p$-linear derivation that maps $d\lambda \mapsto 1$, and by abuse of notation, its extension to $\overline S = S/\pi S = k[x_1, \ldots, x_n]$; see \cite[2.2]{Hochster/Jeffries:2021} for more information on $p$-bases and derivations of inseparability.  By directly checking the universal property, we may verify that the derivation ${d_{\overline S/k^{p}}: \overline S \to \Omega_{\overline S/k^p}}$ where ${d_{\overline S/k^{p}} =  \sum_{\lambda \in \Lambda} \frac{\partial}{\partial \lambda} + \sum_{i=1}^n \frac{\partial}{\partial x_i}}$ represents the functor ${\Der_{k^p}(\overline S, -): \overline S\text{-Mod} \to \overline S\text{-Mod}}$. In fact, it represents the functor ${\Der_{\Z}(\overline S, -): \overline S\text{-Mod} \to \overline S\text{-Mod}}$ due to the following proposition. 

\begin{proposition} \label{Z=k^p}
    Suppose $k$ is a field of characteristic $p > 0$. Then for any $k[x_1, \ldots, x_n]$-module, $\Der_{\Z}(k[x_1, \ldots, x_n], M) = \Der_{k^p}(k[x_1, \ldots, x_n], M)$. 
\end{proposition}

\begin{proof}
    Indeed, as $\Z \to k^p$, any $k^p$-linear derivation is a $\Z$-linear derivation, so $\supseteq$ is clear. Conversely, if $d: k[x_1, \ldots, x_n] \to M$ is a $\Z$-linear derivation, for $c \in k$ and $f \in k[x_1, \ldots, x_n]$, we have $d(c^p f) = c^pd(f) + f p d(c^{p-1}) = c^p d(f)$ as $p$ kills $M$. So $d$ is a $k^p$-linear derivation.
\end{proof}

\noindent Assume further that $V$ is ramified, i.e. $p \in (\pi)^2$. As in Section \ref{sec1}, we will consider $\overline S = S/(\pi)S$ and the functor $\Der_{\Z}(S, -): \overline S\text{-Mod} \to \overline S\text{-Mod}$.

\begin{lemma}\label{lem}
    Let $S = V[x_1, \ldots, x_n]$. The natural map $S/\pi^2S \onto S/\pi S$ splits as map of $\Z$-algebras. 
\end{lemma}
\begin{proof}
    As $p \in (\pi)^2$, $V/\pi^2$ is an artinian, and therefore complete, ring of characteristic $p > 0$. By Cohen's structure theorem, it will contain a copy of its residue field $V/\pi$. In other words, the map $V/\pi^2 \onto V/\pi$ splits as a map of $\Z$-algebras. Tensoring with $\Z[x_1, \ldots, x_n]$, we see that $S/\pi^2S \onto S/\pi S$ splits as a map of $\Z$-algebras. 
\end{proof}

\begin{proposition} \label{directSumProp}
    Suppose $(V, \pi)$ is a ramified discrete valuation ring of mixed characteristic $(0, p)$ with uniformizer $\pi$. Set $S = V[x_1, \ldots, x_n]$ and $\overline S = S/(\pi)S$. Given any $\overline S$-module $M$, we have an isomorphism of $\overline S$-modules \[ \Der_{\Z}(S, M) \cong \Der_{\Z}(\overline S, M) \oplus \Hom_{\overline S}\left( \frac{\pi S}{\pi^2 S}, M\right) \]
\end{proposition}

\begin{proof}
    Given Lemma \ref{lem}, this follows by applying Proposition \ref{thm1} to maps ${\Z \to S \to \overline S}$. 
\end{proof}

\begin{corollary}\label{d/dpi}
    Suppose $(V, \pi)$ is a ramified discrete valuation ring of mixed characteristic $(0, p)$ with uniformizer $\pi$. Set $S = V[x_1, \ldots, x_n]$ and $\overline S = S/(\pi)S$. There exists a $\Z$-derivation $\frac{\partial}{\partial \pi}: S \to \overline S$ such that $\frac{\partial}{\partial \pi} (\pi) = 1$. 
\end{corollary}
\begin{proof} 
    We apply Proposition \ref{directSumProp} to $M = \overline S$. Then the $\overline S$-linear map $\frac{\pi S}{\pi^2 S} \to \overline S$ where $\bar \pi \mapsto 1$ lifts to a $\Z$-derivation $S \to \overline S$ with the desired property.  
\end{proof}

\begin{remark}\label{moduleDeri}
    We now explain the derivation $\frac{\partial}{\partial \pi}$ that is a lift of the map $\bar \pi \mapsto 1$ in \cref{d/dpi} in more detail. Let $\overline{S} = S/\pi S$. Given \cref{lem}, we have $ \sigma: S/\pi^2 S \twoheadrightarrow \overline{S}$ and its splitting $\iota: \overline{S} \xhookrightarrow{} S/\pi^2S$. Following the proof of \cref{thm1}, the map $D = \iota \circ \sigma - \id_{S/\pi^2 S} : S/\pi^2 S \to \pi S/\pi^2 S $ is an additive derivation. Note that $S/\pi^2 S = (S/\pi^2 S)[x_1, \ldots, x_n]$, so an arbitrary $f \in S/\pi^2 S$ can be represented as a finite sum of monomials of the type $u\pi^k x_1^{i_1}\ldots x_n^{i_n}$ where $u \in V^{\times}$ and $k \in \{0, 1\}$. Now one can observe that $D(ux_1^{i_1}\ldots x_n^{i_n}) = 0$ and $D(u\pi x_1^{i_1}\ldots x_n^{i_n}) = -u\pi x_1^{i_1}\ldots x_n^{i_n}$. 
    
    The derivation $D$ induces a split injection of $\overline{S}$-modules $\Hom_{\overline S}\left(\frac{\pi S}{\pi^2 S}, \overline S \right) \to \Der_{\Z}(S, \overline{S})$ where $s \mapsto -s \circ D $. As an $\overline{S}$-module, $\pi S/\pi^2 S \cong S/\pi S$ where $\pi \mapsto 1$. Therefore, \[ \Hom_{\overline S}\left(\frac{\pi S}{\pi^2 S}, \overline S \right) \cong \Hom_{\overline S}(\overline S, \overline S)\] where $\cdot s: \overline S \to \overline S$ for $s \in \overline S$ corresponds to a map $\pi S/\pi^2 S \to \overline S$ where ${\pi \mapsto s}$. The derivation $d: S \to \overline S$ corresponding to this map is such that $d(ux_1^{i_1}\ldots x_n^{i_n}) = 0$ and $d(u\pi x_1^{i_1}\ldots x_n^{i_n}) = s \cdot u x_1^{i_1}\ldots x_n^{i_n}$. The derivation of interest is the derivation $d = \frac{\partial}{\partial \pi}$ corresponding to the linear map where $s = 1$. 
    
    We also note that given any $\overline S$-module $M$ and any $m \in M$,  we have a derivation $m \frac{\partial}{\partial \pi}$ defined as $s \mapsto \frac{\partial (s)}{\partial \pi} m$. Indeed, this is well defined as $M$ is a $\overline S$-module and since $\frac{\partial}{\partial \pi}$ is a derivation, it is easy to check that $m\frac{\partial}{\partial \pi}$ is a derivation too and we have $m \frac{\partial}{\partial \pi}(\pi) = m$. 

\end{remark}

\begin{theorem} \label{universalDer}
     Suppose $(V, \pi)$ is a ramified discrete valuation ring of mixed characteristic $(0, p)$ with uniformizer $\pi$. Set $S = V[x_1, \ldots, x_n]$ and $\overline S = S/\pi S$. Let $\Omega^{\pi}_{S/\Z} = \Omega_{\overline S/k^p} \oplus \overline S d\pi$ and $d^{\pi}_{S/\Z}: S \to \Omega^{\pi}_{S/\Z}$ to be the map $s \mapsto d_{\overline S/k^p} (s) + \frac{\partial(s)}{\partial \pi}d\pi$. Then $d^{\pi}$ represents the functor ${\Der_{\Z}(S, -): \overline S\text{-Mod} \to \overline S\text{-Mod}}$. 
\end{theorem}
\begin{proof}
     Consider some $\overline S$-module $M$ and some arbitrary $\Z$-derivation $d: S \to M$. Suppose $d(\pi) = m$ for some $m \in M$. Take the derivation $m \frac{\partial}{\partial \pi}$ as in Remark \ref{moduleDeri} and so $d - m \frac{\partial}{\partial \pi} : S \to M$ is such that $(d - m \frac{\partial}{\partial \pi})(\pi) = 0$. Therefore, given the direct sum decomposition in Proposition \ref{directSumProp}, $d - m \frac{\partial}{\partial \pi}$ corresponds uniquely to a derivation from $\overline S \to M$. Hence there exists a unique $\overline S$-linear map $\phi_1 : \Omega_{\overline S/k^p} \to M$ such that $\phi_1 \circ d_{\overline S/k^p} = d - m \frac{\partial}{\partial \pi}$ by the universality of $d_{\overline S/k^p}$ and noting Proposition \ref{Z=k^p}. We extend $\phi_1$ to $\overline S$-linear map $\Omega^{\pi}_{S/\Z} \to M$ by setting $\phi_1(d\pi) = 0$. On the other hand let $\phi_2: \Omega^{\pi}_{S/\Z} \to M$ be the $\overline S$-linear map where $\phi_2(d\pi) = m$ and $\phi_2(dy) = 0$ for all basis elements $dy$ of $\Omega_{\overline S/k^p}$. Set $\phi = \phi_1 + \phi_2 : \Omega^{\pi}_{S/\Z} \to M$. Then $\phi$ is a $\overline S$-linear map such that $\phi \circ d^{\pi}_{S/\Z} = d$. If $\phi'$ is any other $\overline S$-linear map such that $\phi' \circ d^{\pi}_{S/\Z} = d$, then $(\phi - \phi')\circ d^{\pi}_{S/\Z} = 0$. Since there is a $b \in S$ such that $d^{\pi}_{S/\Z}(b) = db$ for every basis element $db$ of $\Omega^{\pi}_{S/\Z}$, we have $\phi(db) = \phi'(db)$ for all basis elements, whence $\phi = \phi'$. This completes our proof. 
\end{proof}

\begin{corollary} \label{equivUniversalMaps}
    Suppose $(V, \pi)$ is a ramified discrete valuation ring of mixed characteristic $(0, p)$ with uniformizer $\pi$. Set $S = V[x_1, \ldots, x_n]$ and $\overline S = S/(\pi)S$. The following two derivations are equivalent: 
    \[ (\overline d_{S/\Z}: S \to \overline \Omega_{S/\Z}) \equiv (d^{\pi}_{S/\Z}: S \to \Omega^{\pi}_{S/\Z})\]In particular, $\overline \Omega_{S/\Z} \cong \Omega_{\overline S/k^p} \oplus \overline S d\pi$. 
\end{corollary}

\begin{proof}
Since $\overline d_{S/\Z}$ also represents the functor $\Der_{\Z}(S, -): \overline S\text{-Mod} \to \overline S\text{-Mod}$ (see Section \ref{sec1}), the assertion follows by Theorem \ref{universalDer} and the uniqueness of representing a functor. 
\end{proof}

\begin{proposition} \label{mainEx}
    Suppose $A = \Z $ and $(V, \pi)$ is a ramified discrete valuation ring of mixed characteristic $(0, p)$ with a uniformizer $\pi$. Set $S = V[x_1, \ldots, x_n]$. For $\ma_S = (\pi)S$, $S$ satisfies the conditions of Setup \ref{setup}, i.e. for all $\q \in \Spec S$ such that $\pi \in \q$, we have $S_{\q}$ is a regular local ring such that $S_{\q}/\q^2 S_{\q} \onto S_{\q}/\q S_{\q}$ splits as a map of $\Z$-algebras. 
\end{proposition}

\begin{proof}
    As $S$ is regular, $S_{\q}$ is also regular. Now take $\q \in \Spec S$ such that it contains $\pi$. Then $S_{\q}$ is a ramified local ring, i.e. $p \in \q^2S_{\q}$. Indeed, if we assume otherwise, $S_{\q}/(p)S_{\q}$ is a regular local ring. But since $(V, \pi)$ is a ramified discrete valuation ring, in this quotient, $\overline \pi \not= 0$ and $\overline \pi^n = 0$ for some $n > 1$. This is a contradiction. Therefore, $p \in \q^2S_{\q}$ which means $S_{\q}/\q^2S_{\q}$ is an artinian local ring of characteristic $p > 0$. Since artinian local rings are complete, by Cohen's structure theorem, this ring contains a copy of its residue field $S_{\q}/\q S_{\q}$. In other words, the following map \[ \frac{S_{\q}}{{\q}^2S_{\q}} \onto \frac{S_{\q}}{{\q}S_{\q}}\] splits as a map of $\Z$-algebras. 
\end{proof}

\noindent {\em Mixed Jacobian Matrix: } Suppose $(V, \pi)$ is a ramified discrete valuation ring of mixed characteristic $(0, p)$ and $S = V[x_1, \ldots, x_n]$. Due to Proposition \ref{mainEx}, we can apply Theorem \ref{detectingNonsingularity} to detect nonsingularity of a prime ideal of a finitely generated $V$-algebra that contains the element $\pi$ or equivalently the prime integer $p$. On the other hand, for prime ideals that do not contain $\pi$, the usual Jacobian criterion suffices. 

\begin{remark}
    Suppose $I = (f_1, \ldots, f_t) \subseteq S$ is an ideal and let $R = S/I$. Suppose $\q \in \Spec R \setminus V(\pi) = \Spec R \setminus V(p)$. Then $R_{\q} = V[\pi^{-1}][x_1, \ldots, x_n]_{\q}/I$. Note that $V[\pi^{-1}]$ is a characteristic $0$ field. So the regularity of $R_{\q}$ is detected by the usual Jacobian criterion. 
\end{remark}

\begin{definition} \label{mixedJacobianMatrix}
    Let $I = (f_1, \ldots, f_t) \subseteq S$ be an ideal and set $R = S/I$. We have the ideal $\ma_S = (\pi)S$ and its image $\ma = (\pi)R$ in $R$, and as usual we set $\overline S = S/\ma_S$ and $\overline R = R/\ma$. The ``mixed" Jacobian matrix, $\J^{\pi}_{R}$, is defined to be the matrix over $\overline R$ with rows: 
\begin{center}
    \begin{itemize}
    \item $[ \frac{\partial(f_1)}{\partial \pi} \, \cdots \, \frac{\partial(f_t)}{\partial \pi} ]$
    \item $[ \frac{\partial(f_1)}{\partial x_i} \, \cdots \, \frac{\partial(f_t)}{\partial x_i} ]$ for $i = 1, \ldots, n$. 
    \item $[ \frac{\partial(f_1)}{\partial \lambda}\, \cdots \, \frac{\partial(f_t)}{\partial \lambda} ]$ for $\lambda \in \Lambda$ for a $p$-base $\Lambda$ of $V/(\pi)$.  
\end{itemize}
\end{center}
\end{definition}

\begin{theorem} \label{JacobianSees}
    Suppose $R = S/I$ for some ideal $I \subseteq S$ and $\q \in \Spec R$ such that $\pi \in \q$. Let $h = \height IR_{\q}$ and let $\q$ also denote the image of $\q$ in $R/\pi R$. Then $R_{\q}$ is a regular local ring if, and only if, $\q$ does not contain all $h \times h$ minors of $\J^{\pi}_{R}$. 
\end{theorem}
\begin{proof}
    $R_{\q}$ is a regular local ring if, and only if, ${\rank_{k(\q)}(\J \otimes_{\overline R} k(\q)) = \height IR_{\q}}$ where $\J$ is the $\overline R$-linear map ${\frac{I}{I^2+I\ma} \xhookrightarrow{\J} \overline R \otimes_{\overline S} \overline \Omega_{S/A}}$ where ${\J ([f]) = 1 \otimes \overline d_{S/\Z}(f)}$ for ${f \in I}$; see Theorem \ref{detectingNonsingularity}. Due to Corollary \ref{equivUniversalMaps}, we have ${\overline d_{S/\Z} = d^{\pi}_{S/\Z}}$ as defined in Theorem \ref{universalDer}. Now it follows that the image of $\J$ is the same as the image of $\J^{\pi}_{R}$. Now replacing $\J$ by $\J^{\pi}_{R}$, the assertion follows as ${\rank_{k(\q)}(\J^{\pi}_{R} \otimes_{\overline R} k(\q)) = \height IR_{\q}}$ if, and only if, there exist a non-vanishing $h \times h$ minor of $\J^{\pi}_{R}$ over $k(\q)$ which is equivalent to saying  $\q$ does not contain all $h \times h$ minors of $\J^{\pi}_{R}$. 
\end{proof}

\begin{remark} \label{rank<=h}
   For any minimal prime $\q$ of an ideal $I = (f_1, \ldots, f_t) \subseteq S$, we have $\rank_{k(\q)}(\J^{\pi}_R \otimes_{\overline R} k(\q)) \leq h = \height I$. Indeed, localizing at $\q$, $S_{\q}$ is a regular local ring of dimension $h$. Then we have $\q S_{\q} = (g_1, \ldots, g_h)S_{\q}$ minimally. As $IS_{\q} \subseteq \q S_{\q}$, we have that $f_i = \sum_{k=1}^h r_{ik}g_k$ in $S_{\q}$ for all $i$. We have by the Leibniz rule $\partial f_i/\partial \pi = \sum_{k=1}^h r_{ik} (\partial g_k/ \partial \pi) + \sum_{k=1}^h (\partial r_{ik} / \partial \pi)  g_k$ for $i = 1, \ldots, t$,  $\partial f_i/\partial x_j = \sum_{k=1}^h r_{ik} (\partial g_k/ \partial x_j) + \sum_{k=1}^h (\partial r_{ik} / \partial x_j)  g_k$ for $i = 1, \ldots, t$ and $j = 1, \ldots, n$ and  $\partial f_i/\partial \lambda = \sum_{k=1}^h r_{ik} (\partial g_k/ \partial \lambda) + \sum_{k=1}^h (\partial r_{ik} / \partial \lambda)  g_k$ for $i = 1, \ldots, t$ and $\lambda \in \Lambda$ for $\Lambda$ a $p$-base of $V/\pi V$. Now modulo $\q S_{\q}$, the second summand goes away in all cases. Therefore, $J^{\pi}_R \otimes_{\overline R} k(\q) = [\bar r_{ij}] J'^{\pi}$, where $J'^{\pi}$ denotes the mixed jacobian matrix on the generators of $\q S_{\q}$. Note then that $J'^{\pi}$ is a matrix with $h$ rows, so its $h+1$ minors are $0$. This implies that $h+1$ minors of $\J^{\pi}_R \otimes_{\overline R} k(\q)$ is also $0$ as this matrix factors through a $k(\q)$-vector space of dimension $h$. Therefore, $\rank_{k(\q)}(\J^{\pi}_R \otimes_{\overline R} k(\q)) \leq h$.
\end{remark}
\begin{corollary} \label{JacobianCriterion}
     Suppose $R = S/I$ for some ideal $I \subseteq S$ such that $I$ is of pure height $h$ (i.e. all minimal primes of $I$ have the same height and $I$ has no embedded primes). Then \[ \Sing R \cap \V(p) = \V(I_{h}(\J^{\pi}_R)) \subseteq \Spec(R/(\pi)R) \]where we identify $V(\pi) = V(p) \subseteq \Spec R$ with $\Spec(R/(\pi) R)$ and $I_h(\J^{\pi}_R)$ denotes the ideal generated by $h \times h$ minors of $\J^{\pi}_R$ in $R/(\pi)R$. In particular, $\Sing R \cap \V(p)$ is a closed subset of $\Spec R$. 
\end{corollary} 
\begin{proof}
     Note that for any $\q \in \Spec S$, $\height IS_{\q} = h$ as $I$ is of pure height $h$. By Remark \ref{rank<=h}, for minimal primes $\q$ over $I$, $\rank_{k(\q)}(\J^{\pi}_R \otimes_{\overline R} k(\q)) \leq h$. Since rank can only go down going modulo larger primes, $\rank_{k(\q)}(\J^{\pi}_R \otimes_{\overline R} k(\q))\leq h$ for all $\q \supseteq I$.  Now for a prime ideal $\q \in \Spec S$ such that $ \q \supseteq I + (\pi)S$, equivalently $\q \in \Spec R \cap \V(\pi) = \Spec R \cap \V(p)$, we have by Theorem \ref{JacobianSees}, $R_{\q}$ is singular (i.e. not a regular local ring) if, and only if, $I_h(J^{\pi}_R) \subseteq \q$ in $R/(\pi)R$. Therefore, the assertion follows. 
\end{proof}

\begin{remark} \label{HJwork}
    Noting \cref{moduleDeri}, given a presentation for a finitely generated algebra over a ramified discrete valuation ring, \cref{JacobianCriterion} gives an easy way to compute its singular locus of primes containing the prime integer $p$. We now explain how the results of Hochster-Jeffries and Saito apply to compute singular loci of algebras over ramified discrete valuation rings. 
    
    Suppose $V$ is a ramified discrete valuation of mixed characteristic with uniformizer $\pi$. We claim that $V$ does not admit a $p$-derivation modulo $p^2$. For sake of contradiction, assume $\delta_p: V/p^2V \to V/pV$ is a $p$-derivation modulo $p^2$. Then $\delta_p(p)$ is a unit. However, as $V$ is ramified $(\overline \pi)$ is a nonzero maximal ideal of $V/pV$ and since $p \in (\pi)^2$, $\delta_p(p) \subseteq \delta_p(\pi)(\overline{\pi})$ which is a contradiction. So the results in \cite{Hochster/Jeffries:2021, Saito:2022} cannot be extended directly. 
    
    Suppose $V$ is a finite algebra over an unramified discrete valuation subring, i.e. \[ V \cong U[x_1, \ldots, x_n]/I\]where $(U, pU)$ is an unramified discrete valuation ring. In this case, any finitely generated $V$ algebra can be viewed as a finitely generated $U$ algebra. If we can find such a presentation for $V$, then the criterion of Hochster-Jeffries and Saito applies. For example, if $(V, \pi)$ is a $\pi$-complete discrete valuation ring, then by Cohen's structure theorem, $V$ admits the desired presentation. However, given an arbitrary ramified discrete valuation $V$, it is not easy to write down the presentation as above, even for a $\pi$-complete $V$. In this sense, \cref{JacobianCriterion} gives a simpler alternative. On the other hand, we suspect there exists a non-complete ramified discrete valuation ring that is not finitely generated over any of its unramified discrete valuation subring, but we have not been able to find such an example in the literature. 
    
    \end{remark}


\begin{remark}\label{calculation}
    Suppose $(V, \pi)$ is a ramified discrete valuation ring with uniformizer $\pi$ such that its residue field $k = V/\pi V$ is $F$-finite. Set $S = V[x_1, \ldots, x_n]$ and fix $\p \in \Spec S \cap \V(\pi)$. Let $a = \log_p[k: k^p]$ and $b = \log_p[k(\p): k(\p)^p]$. We want to show that $\dim S + a = \height \p + b$. This argument is carried out in proof of \cite[Theorem 4.13]{Hochster/Jeffries:2021} but their argument has a small typo so we spell out all the details here. We have $\overline S = S/\pi S = k[x_1, \ldots, x_n]$ is an $F$-finite ring of characteristic $p > 0$, so given prime ideals $\q \subseteq \q'$ in $\overline S$, we have the formula $\height(\q'/\q) = \log_p[k(\q): k(\q)^p]-\log_p[k(\q'):k(\q')^p]$ \cite{Kunz:1976}. Applying this formula for $(0) \subseteq (x_1, \ldots, x_n)$, note that we get $n + a = \log_p[\text{Frac}(\overline S) - \text{Frac}(\overline S)^p]$. Then applying it with $(0) \subseteq \overline \p = \p/(\pi)$, we get $\height(\overline \p) = (n+a) - \log_p[k(\p): k(\p)^p] = (n+a) - b$, noting that $k(\p) = k(\overline \p)$ as $\pi \in \p$. Finally, since $\height(\overline \p) = \height \p - 1$ and $\dim S = n + 1$, we have $\dim S + a = \height \p + b$. 
\end{remark}

\begin{corollary} \label{non-free locus}
    Let $(V, \pi)$ be a ramified discrete valuation ring of mixed characteristic $(0, p)$ with uniformizer $\pi$ such that $k = V/\pi V$ is an $F$-finite field. Suppose $(R, \m, K)$ is a local algebra of essentially finite type over $V$ such that $p \in \m$. Set $\overline R = R/(\pi)R$ and $\overline \Omega_{R/\Z} = \Omega_{R/\Z}/(\pi)\Omega_{R/\Z}$. Then $R$ is a regular local ring if, and only if, $\overline \Omega_{R/\Z}$ is a free $\overline R$ module of rank $\dim R + \log_p[K : K^p]$. 
\end{corollary}

\begin{proof}

    We closely follow the proof of \cite[Theorem 4.13]{Hochster/Jeffries:2021}. Suppose $R = (S/I)_{\p}$ for $S = V[x_1, \ldots, x_n]$ and $I = (\mathbf{f}) = (f_1, \ldots, f_t) \subseteq S$ an ideal such that $\p \supseteq I$. Let $\J^{\pi}(\mathbf{f})$ be the mixed Jacobian matrix on the generators of $I$ as in Definition \ref{mixedJacobianMatrix} considered as a matrix over $\overline R$. We have by Proposition \ref{2ndseq} and Corollary \ref{equivUniversalMaps} that $\J^{\pi}(\mathbf{f})$ is a presentation matrix for $\overline \Omega_{R/\Z}$. In particular, $\overline \Omega_{R/\Z}$ is generated by $\dim S + a$ elements where $a = \log_p[k : k^p]$ which is the size of a $p$-basis of $k$. Set $b = \log_p[K : K^p]$. We have $\dim S + a = \height \p + b = \dim S_{\p} + b = \height I_{\p} + \dim R + b $ by the calculation in Remark \ref{calculation}. Now by the characterization of locally free modules via Fitting ideals, $\overline \Omega_{R/\Z}$ is locally free of rank $\dim R + b$ if and only if $F_{\dim R + b}(\overline \Omega_{R/\Z}) = \overline R$ and $F_{< \dim R + b}(\overline \Omega_{R/\Z}) = 0$. Using the fact that $\J^{\pi}(\mathbf{f})$ is a presentation matrix for $\overline \Omega_{R/\Z}$ and that $(\dim S + a) - (\dim R + b) = \height I $, this translates to saying $I_{\height I}(\J^{\pi}(\mathbf{f})) = \overline R$ and $I_{> \height I}(\J^{\pi}(\mathbf{f})) = 0$. But this is if and only if $\rank(J^{\pi}_{S/I} \otimes_{S/I} k(\p)) = \height I$ which by Theorem \ref{JacobianSees} is equivalent to $(S/I)_{\p} = R$ being a regular local ring.  
\end{proof}


\bibliographystyle{amsplain}
\bibliography{references}

\end{document}